\documentclass[11pt]{amsart}

\usepackage[foot]{amsaddr}

\usepackage{amssymb,amsfonts,amsmath,graphicx,verbatim,eufrak,amsthm,calc}
\usepackage{color}
\usepackage{xcolor}
\usepackage{newfloat}
\usepackage{hyperref}
\usepackage{faktor}

\usepackage{tikz}
\hypersetup{
hidelinks,backref=true,pagebackref=true,hyperindex=true,
    colorlinks=true,       
    linkcolor=violet,          
    citecolor=green,        
    filecolor=magenta,      
    urlcolor=cyan           
}

\colorlet{genial}{black} 
\colorlet{genialsol}{black}

\makeatletter

\newtheoremstyle{genialnumbox}
{7pt}
{7pt}
{\normalfont}
{}
{\small\bf\sffamily\color{genial}}
{\;}
{0.25em}
{%
{\small\sffamily\color{genial}\thmname{#1}}%
{\nobreakspace\thmnumber{\@ifnotempty{#1}{}\@upn{#2}}}
\thmnote{{\nobreakspace\the\thm@notefont\sffamily\bfseries\color{black}\nobreakspace(#3)}} 
}

\newtheoremstyle{blacknumex}
{7pt}
{7pt}
{\normalfont}
{} 
{\small\bf\sffamily}
{\;}
{0.25em}
{%
{\small\sffamily\color{genial}\thmname{#1}}%
{\nobreakspace\thmnumber{\@ifnotempty{#1}{}\@upn{#2}}}
\thmnote{{\nobreakspace\the\thm@notefont\sffamily\bfseries\color{black}\nobreakspace(#3)}} 
}

\newtheoremstyle{blacknumbox} 
{7pt}
{7pt}
{\normalfont}
{}
{\small\bf\sffamily}
{\;}
{0.25em}
{%
{\small\sffamily\color{genial}\thmname{#1}}%
{\nobreakspace\thmnumber{\@ifnotempty{#1}{}\@upn{#2}}}
\thmnote{{\nobreakspace\the\thm@notefont\sffamily\bfseries\color{black}\nobreakspace(#3)}} 
}

\newcommand\fH[1]{\sbox0{#1}\dimen0=\ht0 \advance\dimen0 -1ex
  \sbox2{\'{}}\sbox2{\raise\dimen0\box2}%
  {\ooalign{\hidewidth\kern.1em\copy2\kern-.5\wd2\box2\hidewidth\cr\box0\crcr}}}

\newtheoremstyle{genialnum}
{7pt}
{7pt}
{\normalfont}
{}
{\small\bf\sffamily\color{genial}}
{\;}
{0.25em}
{%
{\small\sffamily\color{genial}\thmname{#1}}%
{\nobreakspace\thmnumber{\@ifnotempty{#1}{}\@upn{#2}}}
\thmnote{{\nobreakspace\the\thm@notefont\sffamily\bfseries\color{black}\nobreakspace(#3)}} 
}

\makeatother

\RequirePackage[framemethod=default]{mdframed} 

\newmdenv[skipabove=7pt,
skipbelow=7pt,
rightline=false,
leftline=false,
topline=false,
bottomline=false,
backgroundcolor=black!5,
linecolor=genial,
innerleftmargin=5pt,
innerrightmargin=5pt,
innertopmargin=10pt,
leftmargin=0cm,
rightmargin=0cm,
innerbottommargin=10pt]{tBox}

\newmdenv[skipabove=7pt,
skipbelow=7pt,
rightline=false,
leftline=false,
topline=false,
bottomline=false,
backgroundcolor=genial!10,
linecolor=genial,
innerleftmargin=5pt,
innerrightmargin=5pt,
innertopmargin=5pt,
innerbottommargin=5pt,
leftmargin=0cm,
rightmargin=0cm,
linewidth=4pt]{eBox}	

\newmdenv[skipabove=7pt,
skipbelow=7pt,
rightline=false,
leftline=true,
topline=false,
bottomline=false,
linecolor=genial!50,
innerleftmargin=5pt,
innerrightmargin=5pt,
innertopmargin=5pt,
leftmargin=0cm,
rightmargin=0cm,
linewidth=4pt,
innerbottommargin=5pt]{dBox}	

\newmdenv[skipabove=7pt,
skipbelow=7pt,
rightline=false,
leftline=false,
topline=false,
bottomline=false,
linecolor=gray,
backgroundcolor=black!5,
innerleftmargin=5pt,
innerrightmargin=5pt,
innertopmargin=5pt,
leftmargin=0cm,
rightmargin=0cm,
linewidth=4pt,
innerbottommargin=5pt]{cBox}

\newmdenv[skipabove=7pt,
skipbelow=7pt,
rightline=false,
leftline=false,
topline=false,
bottomline=false,
linecolor=gray,
backgroundcolor=black!5,
innerleftmargin=5pt,
innerrightmargin=5pt,
innertopmargin=5pt,
leftmargin=0cm,
rightmargin=0cm,
linewidth=4pt,
innerbottommargin=5pt]{pBox}

\newmdenv[skipabove=7pt,
skipbelow=7pt,
rightline=false,
leftline=false,
topline=false,
bottomline=false,
linecolor=genialsol,
innerleftmargin=5pt,
innerrightmargin=5pt,
innertopmargin=0pt,
leftmargin=0cm,
rightmargin=0cm,
linewidth=4pt,
innerbottommargin=0pt]{solBox}

\theoremstyle{genialnumbox}
\newtheorem{thm1}{Theorem}[section]
\newtheorem{ithm1}[thm1]{$\star$ THEOREM}
\newtheorem{ques1}[thm1]{Question}
\newtheorem{conj1}[thm1]{Conjecture}

\theoremstyle{blacknumex}
\newtheorem{exer}[thm1]{Exercise}
\newtheorem{exer*}[thm1]{$\ast$ Exercise}

\theoremstyle{blacknumbox}
\newtheorem{dfn1}[thm1]{Definition}

\theoremstyle{genialnum}
\newtheorem{cor1}[thm1]{Corollary}
\newtheorem{prop1}[thm1]{Proposition}
\newtheorem{lem1}[thm1]{Lemma}

\newtheorem{exm1}[thm1]{Example}

\newenvironment{thm}{\paragraph{ } \begin{tBox}\begin{thm1}}{\end{thm1}\end{tBox}}

\newenvironment{exe*}{\paragraph{ } \begin{eBox}\begin{exer*}}{\hfill{\color{genial}
\ensuremath{\diamond\diamond\diamond}}\end{exer*}\end{eBox}}
\newenvironment{dfn}{\paragraph{ } \begin{dBox}\begin{dfn1}}{\end{dfn1}\end{dBox}}	

\newenvironment{cor}{\paragraph{ } \begin{cBox}\begin{cor1}}{\end{cor1}\end{cBox}}	
\newenvironment{ques}{\paragraph{ } \begin{cBox}\begin{ques1}}{\end{ques1}\end{cBox}}

\newenvironment{prop}{\paragraph{ } \begin{pBox}\begin{prop1}}{\end{prop1}\end{pBox}}	
\newenvironment{lem}{\paragraph{ } \begin{pBox}\begin{lem1}}{\end{lem1}\end{pBox}}

\newenvironment{lem*}[1]{\vspace{1ex}\noindent
{\bf Lemma* (#1).} [restatement]  \hspace{0.5em} \em }{ }

\newenvironment{thm*}[1]{\begin{cBox}
\vspace{1ex}\noindent 
{\bf Theorem* (#1).} [restatement]  \hspace{0.5em} }{\end{cBox}}

\theoremstyle{genialnum}

\newtheorem*{clm*}{Claim}

\newenvironment{sol}%
{\begin{solBox}
\par \noindent 
\scriptsize
{\bf Solution to ex:{\color{blue} \arabic{exer}}.}  {\color{red} \ \  :( } \\ }%
{\hfill {\color{blue} :) $\checkmark$} \end{solBox}}

\newcommand{\ENDEXER}{
{\expandafter\comment}
{\expandafter\endcomment}
}

\newtheorem{remark}[thm1]{Remark}

\makeatletter
\renewcommand{\@seccntformat}[1]{\llap{\textcolor{genial}{\csname the#1\endcsname}\hspace{1em}}}                    
\renewcommand{\section}{\@startsection{section}{1}{\z@}
{-4ex \@plus -1ex \@minus -.4ex}
{1ex \@plus.2ex }
{\normalfont\large\sffamily\bfseries}}
\renewcommand{\subsection}{\@startsection {subsection}{2}{\z@}
{-3ex \@plus -0.1ex \@minus -.4ex}
{0.5ex \@plus.2ex }
{\normalfont\sffamily\bfseries}}
\renewcommand{\subsubsection}{\@startsection {subsubsection}{3}{\z@}
{-2ex \@plus -0.1ex \@minus -.2ex}
{.2ex \@plus.2ex }
{\normalfont\small\sffamily\bfseries}}                        
\renewcommand\paragraph{\@startsection{paragraph}{4}{\z@}
{-2ex \@plus-.2ex \@minus .2ex}
{.1ex}
{\normalfont\small\sffamily\bfseries}}

\newcommand{\nicetitle}[2]{\title[#2]{\sffamily\large #1} }
\newcommand{\niceauthor}[1]{\author{\sffamily #1} }

\makeatother

\newcommand{\Integer}{\mathbb{Z}}

\newcommand{\Z}{\Integer}
\newcommand{\N}{\mathbb{N}}

\newcommand{\eps}{\varepsilon}

\newcommand{\ie}{{\em i.e.\ }}
\newcommand{\eg}{{\em e.g.\ }}

\newcommand{\IT}{\mathrm{IT}}

\def\squareforqed{\hbox{\rlap{$\sqcap$}$\sqcup$}}
\def\qed{\ifmmode\squareforqed\else{\unskip\nobreak\hfil
\penalty50\hskip1em\null\nobreak\hfil\squareforqed
\parfillskip=0pt\finalhyphendemerits=0\endgraf}\fi}

\newcommand{\ignore}[1]{ }

\newcommand{\vphi}{\varphi}

\newcommand{\AND}{\qquad \textrm{and} \qquad}

\newcommand{\define}[1]{\textbf{#1}}

\newcommand{\USS}{\mathsf{USS}}

\newcommand{\CS}{\mathsf{CS}}

\nicetitle{Completely Syndetic Sets in Discrete Groups}{Completely syndetic sets in Discrete Groups}

\niceauthor{Guy Salomon}
\address{GS: School of Mathematical Sciences\\ Holon Institute of Technology \\ Holon \\ Israel}
\email{guysal@hit.ac.il}

\niceauthor{Yotam Svoray}
\address{YS: Department of Mathematics\\The University of Utah\\Salt Lake City \\ Utah \\USA}
\email{svoray@math.utah.edu}

\niceauthor{Ariel Yadin}
\address{AY: Department of Mathematics\\Ben-Gurion University of the Negev\\Beer-Sheva \\Israel}
\email{yadina@bgu.ac.il}
\thanks{This research partially supported by 
the Israel Science Foundation, grant no. 954/21.  YS thanks Ben-Gurion University of the Negev for hospitality while working on this project.}

\begin{document}

\maketitle

\begin{abstract}
    We study completely syndetic ($\CS$) sets in discrete  groups --- subsets that for every $n \in \mathbb N$ admit finitely many left translates that jointly cover every $n$-tuple of group elements. While for finitely-generated groups, the non-virtually nilpotent ones admit a partition into two $\CS$ sets, we show that virtually abelian groups do not. We also characterize $\CS$ subsets of $\mathbb{Z}$, and as a result characterize subsets of $\mathbb{Z}$ whose closure in $\beta \mathbb{Z}$ contains the smallest two sided ideal of $\beta \mathbb{Z}$. Finally, we show that $\CS$ sets can have an arbitrarily small density.
\end{abstract}

\section{Introduction}
Let $G$ be a discrete countable group. A subset $A\subset G$ is called (left) {\em syndetic} if there exists finitely many left translates of $A$ whose union covers the entire group --- that is, every $g\in G$ lies in at least one of them. This notion generalizes the idea of having a finite index from subgroups to arbitrary subsets. 

Syndetic sets have appeared in various mathematical contexts over the past century \cite{BalcarFranek97,Fur81, GotHel55, Patil19}, particularly in connection with subsets of the integers. In $\mathbb{Z}$, the syndetic subsets are exactly those with bounded gaps (when viewed as increasing sequences).

It is a fact that every discrete group can be partitioned into two disjoint syndetic sets; see Remark \ref{rem:every_group_is_syndetic_decomposable}. For instance, the even and odd integers form such a pair in $\mathbb{Z}$.

Recently, Kennedy, Raum, and Salomon \cite{KenRaumSal22} introduced a stronger notion of syndeticity in their study of the universal minimal proximal and strongly proximal flows of discrete groups. A subset $A \subset G$ is called (left) completely syndetic (abbreviated $\CS$) if for every $n \in \mathbb{N}$, there exist finitely many left translates of $A$ such that every $n$-elements of $G$ all lie together in at least one of these translates; see Definition \ref{dfn:CS} below.

$\CS$ sets are ``very large'' sets in a certain sense.  It is simple to verify that if $A \subset B$ and $A$ is $\CS$,
then so is $B$. It stands to reason that ``large'' sets cannot fit together in a ``small'' group without overlapping.  
In other words, a natural question is: {\em Which groups can be partitioned into two $\CS$ sets?}
In \cite{KenRaumSal22}, it is shown that groups which {\em cannot} be decomposed this way must be strongly amenable --- a property equivalent, in the finitely generated case, to being virtually nilpotent. 
Thus, ``large'' groups admit disjoint $\CS$ sets.
One of the main motivations for this work was to investigate what happens in the virtually nilpotent case (``small'' groups). While we do not provide a complete answer, we do obtain several partial results and explore related phenomena.

One such result is the following.

\begin{thm}[Theorem \ref{thm:virt_abelian_not_CS-dcmpsbl}]\label{thm:intro_virt_abelian_not_CS-dcmpsbl}
Virtually abelian groups cannot be partitioned into two $\CS$ sets.
\end{thm}

Unfortunately, the case of virtually nilpotent groups that are not virtually abelian remains open.
\begin{ques}[Question \ref{ques:nilpotent}]
Is there a virtually nilpotent group, necessarily not virtually abelian, that can be partitioned into two $\CS$ sets?
\end{ques}

In fact, the {\em virtually} aspect is not a genuine obstacle: a group can be partitioned into two $\CS$ sets if and only if some finite-index subgroup can. Therefore, in the finitely generated case, the question reduces to studying nilpotent groups; see Proposition \ref{prop:finite index}.

Unlike syndetic sets in $\mathbb{Z}$, which are easy to describe, the structure of $\CS$ subsets of $\mathbb{Z}$ is more subtle. One of our aims is to give an explicit characterization of such sets.

To this end, we introduce the notion of a uniform sub-Szemer\'{e}di set. A subset $A \subset \mathbb{Z}$ is called {\em uniform sub-Szemer\'{e}di} if it contains arbitrarily long sub-arithmetic progressions of a uniformly bounded difference;
see Definition \ref{dfn:USS} below. With this terminology, we obtain the following characterization:

\begin{thm}[Theorem \ref{thm:CSinZ}]
$A \subset \Z$ is $\CS$ if and only if $\Z \setminus A$ is not uniform sub-Szemer\'{e}di.
\end{thm}

Any $\CS$ set is in particular thick. A set $A \subset G$
is (left) {\em thick} if for any finite set $F \subset G$ there exists some $g \in  G$ such that $F \subset gA$. One feature of $\CS$ sets is that they are ``so large'' that they are simultaneously syndetic and thick. It is therefore natural to ask whether the converse is also true. Using the above characterization of $\CS$ sets of the integers, we show this is false.

\begin{cor}[Corollary \ref{cor:syndethick_yet_non_CS}]
    There are thick and syndetic sets of integers that are not $\CS$.
\end{cor}

The notion of complete syndeticity therefore expresses a very strong sense of ``largeness'' (at least stronger than that of thickness and syndeticity). 
Another common measure of the largeness of a set is its density. 
We show that these two notions of largeness --- complete syndeticity and large density --- are unrelated.

\begin{thm}[Theorem \ref{thm:small CS in any group}]
\label{thm:intro_small CS in any group}
Let $G$ be an infinite finitely generated group.
For any $\eps>0$ there exists a $\CS$ set $A \subset G$ with a density less than $\eps>0$.
\end{thm}

Let $\beta G$ denote the Stone--\v{C}ech compactification of a countable discrete group $G$. This compactification carries a natural semigroup structure, turning it into a right topological semigroup (see the discussion preceding Proposition~\ref{prop:CS and right ideals}). A subset $A \subset G$ is $\CS$ if and only if its closure $\overline{A} \subset \beta G$ contains a right ideal of $\beta G$ (see Proposition~\ref{prop:CS and right ideals}). In this sense, $\CS$ sets arise naturally in the context of the Stone--\v{C}ech compactification and are well worth studying from this perspective.

Many of the results in this paper can be naturally formulated in this language. For instance, Theorem~\ref{thm:intro_virt_abelian_not_CS-dcmpsbl} translates to the statement that if $G$ is virtually abelian, then $\beta G$ contains a unique minimal closed right ideal, which then must necessarily be equal to $K(\beta G)$, the unique smallest two-sided ideal of $\beta G$. Combined with Theorem~\ref{thm:intro_small CS in any group}, this leads to the following corollary:

\begin{cor}[Corollary~\ref{cor:ol_A_contains_K(beta Z)}]
Let $A \subset \Z$. Then $\overline{A} \supset K(\beta \Z)$ if and only if $\Z \setminus A$ is not uniform sub-Szemer\'{e}di.
\end{cor}

More broadly, many of our results admit natural reformulations in terms of the Stone--\v{C}ech compactification, though we refrain from translating each one explicitly.

Beyond the introduction, the paper is organized into three sections. Section 2 explores the connection between $\CS$ sets and closed right ideals in the Stone--\v{C}ech compactification $\beta G$ of a group $G$. We establish a dichotomy: $\beta G$ always contains either one or infinitely many minimal closed right ideals. In the case of virtually abelian groups, we show that $\beta G$ has exactly one such ideal, which implies that these groups cannot be partitioned into two $\CS$ sets. We also introduce the notion of intersection translates of a minimal closed right ideal and relate it to our central question. Section 3 focuses on $\CS$ subsets of $\mathbb{Z}$, providing a complete characterization. Section 4 demonstrates that $\CS$ sets can have arbitrarily small density.

\section{CS sets and their closures in the Stone--\v{C}ech compactification}

We start with the definition of completely syndetic sets.

\begin{dfn} \label{dfn:CS}
Let $G$ be a discrete group.  A subset $A \subset G$ is \define{$n$-syndetic} if 
there exists a finite set $F_n \subset G$ such that for any subset $S \subset G$ of size $|S| \leq n$
there exists some $f \in F_n$ such that $S \subset f A$.

$A$ is called \define{completely syndetic}, or $\CS$, if $A$ is $n$-syndetic for all $n$.
\end{dfn}

Subsets that are $1$-syndetic are simply known as \define{syndetic}.

\begin{remark} \label{rem:every_group_is_syndetic_decomposable}
Bella and Malykhin \cite{BellaMalykhin99} asked if every discrete group can be partitioned into two syndetic sets. For instance, in the group of integers, the even and odd numbers form such a partition --- both are syndetic, as they are translations of a finite index subgroup. Although Protasov \cite{Protasov03} later provided an ad hoc solution, a more conceptual answer was already implicit in earlier work by Balcar and Franek \cite{BalcarFranek97}, in their study of the universal minimal space of a semigroup.

Let us briefly describe this approach in the setting of groups. A \define{minimal} $G$-space is a compact Hausdorff space equipped with a continuous action of $G$ such that it contains no proper nonempty closed $G$-invariant subsets. Every group $G$ admits a {\em universal} minimal $G$-space --- one through which every other minimal $G$-space factors. If $G$ is countable discrete, then this universal minimal space is nontrivial totally (in fact, extremally) disconnected space.

The collection of clopen subsets of this universal space hence forms a nontrivial $G$-Boolean algebra. It is a fact that this algebra is isomorphic to an invariant Boolean subalgebra of the power set of $G$ (with $G$ acting by left translations), which is maximal with respect to the property that every nontrivial subset in this collection is syndetic. In particular, this implies the existence of a subset $A \subset G$ such that both $A$ and its complement $G \setminus A$ are syndetic. For more details, see \cite{BalcarFranek97}. 

In a similar spirit, Kennedy, Raum, and Salomon studied the universal minimal proximal $G$-space \cite{KenRaumSal22}. A $G$-space is said to be \define{proximal} if for any two points $x$ and $y$ in the space, there exists a net $(g_\alpha)$ in $G$ such that $\lim g_\alpha x = \lim g_\alpha y$. As with the minimal case, there exists a universal minimal proximal $G$-space, which is totally (in fact, extremally) disconnected. The $G$-Boolean algebra of its clopen subsets is isomorphic to a $G$-invariant Boolean subalgebra of the power set of $G$, and is maximal with respect to the property that every nontrivial element is {\em completely syndetic}.

A notable difference from the minimal setting is that this Boolean algebra can be trivial; such groups are called \define{strongly amenable}. This occurs precisely when $G$ has no nontrivial ICC quotient
(\ie no nontrivial quotient with the property that every nontrivial conjugacy class is infinite). In the case where $G$ is finitely generated, this is equivalent to $G$ being virtually nilpotent. For further discussion, see \cite{FriTamVa18}.
\end{remark}

Completely syndetic subsets of a group also arise naturally in the study of the Stone--\v{C}ech compactification of the group. Recall that any group $G$ acts continuously on its Stone--\v{C}ech compactification $\beta G$ from both the left and right. Extending each of these actions to a continuous semigroup action of $\beta G$ on itself yields two distinct semigroup structures on $\beta G$. Here, we fix the first of these structures, under which $\beta G$ forms a {\em right} topological semigroup. By \define{right topological semigroups}, we mean semigroups $S$ equipped with a Hausdorff topology such that for every $s \in S$ multiplication by $s$ from the right, namely the map $t \mapsto ts$, is a continuous function of $t$. For more information on the structure and properties of $\beta G$ as a right topological semigroups see~\cite[Chapter 4]{HindmanStrauss}.

In any semigroup, one may consider right or left ideals. A subset $R \subset \beta G$ is a \define{right ideal} if $Rx := \{ yx \ : \ y \in R \} \subset R$ 
for all $x \in \beta G$.
A right ideal $R$ is called \define{minimal} if it does not strictly contain any other nontrivial right ideal.
Similar concepts exist for left ideals. 

Right topological semigroups that are also compact enjoy the property that their right and left ideals always contain {\em minimal} right and left ideals, respectively. Due to the right continuity, left ideals are automatically closed; right ideals, however, are generally not. Another important property of compact right topological groups $S$ is that they have a unique smallest two-sided ideal, denoted by $K(S)$, which is the union of all minimal left ideals of $S$ and also the union of all minimal right ideals of $S$. For more details on compact right topological semigroups, and proofs of the above facts, 
see \cite[Chapter 2]{HindmanStrauss}.

The following proposition translates the notion of $\CS$ sets from the underlying group $G$ to its Stone--\v{C}ech compactification $\beta G$.

\begin{prop}[Proposition 3.14 of \cite{KenRaumSal22}] 
\label{prop:CS and right ideals}
    A set $A \subset G$ is $\CS$ if and only if its closure $\overline{A}$ in $\beta G$ contains a right ideal.
\end{prop}
As we mentioned above, $\beta G$ has the property that its right ideals contain minimal right ideals (since it is a compact right topological semigroup). In particular, subsets of $G$ are $\CS$ if and only if their closures in $\beta G$ contain a closed minimal right ideal. 

\begin{dfn}
    For an element $x \in \beta G$, the \define{principal closed right ideal} generated by $x$ is defined 
    to be $r(x):=\overline{x \beta G}$.
\end{dfn}
Obviously, closed minimal right ideals are always principal closed right ideals, namely, of the form $r(x)$ for some $x \in K(\beta G)$, where $K(\beta G)$ is the unique smallest two-sided ideal of $\beta G$. In fact, in this case, $x$ can always be taken to be an idempotent (\ie $x^2 = x$),
see \cite[Theorem 2.7]{HindmanStrauss}.

As we explained in Remark \ref{rem:every_group_is_syndetic_decomposable} every discrete countable group can be decomposed into two syndetic sets. But when can a group be decomposed into two $\CS$ sets?  In view of Proposition
\ref{prop:CS and right ideals}, this question automatically translates to the existence of two distinct minimal closed right ideals in $\beta G$.

In fact, the number of minimal closed right ideals exhibits a dichotomy, as we now explain.

\begin{prop}\label{pn:dichotomy}
Let $G$ be a discrete group. Then $\beta G$ has either a unique minimal closed right ideal or infinitely many.
\end{prop}

\begin{proof}
If $G$ is not strongly amenable, then by \cite{KenRaumSal22}, there exist infinitely many disjoint $\CS$ sets in $G$, which implies that $\beta G$ contains infinitely many minimal closed right ideals. Thus, we may assume that $G$ is strongly amenable.

Suppose that $\beta G$ has only finitely many minimal closed right ideals. Let $R$ denote the finite set of these right ideals. 

Let $x_1, x_2, y \in \beta G$ with $r(x_1), r(x_2), r(y) \in R$, and suppose $y = \lim_\alpha g_\alpha$ for some net $(g_\alpha)_{\alpha}$ in $G$. Since $R$ is finite, by passing to a subnet, we may assume without the loss of generality that there are $r(z_1), r(z_2) \in R$ such that $r(g_\alpha x_1) = r(z_1)$ and $r(g_\alpha x_2) = r(z_2)$ for all $\alpha$.

Then, $yx_i = \lim g_\alpha x_i \in r(z_i)$, and thus $r(yx_i) \subset r(z_i) \cap r(y)$, for each $i=1,2$. By minimality, this implies $r(z_1) = r(z_2) = r(y)$. Hence, equipping $R$ with the discrete topology and with the continuous action of $G$ by left translations, we see that the action of $G$ on $R$ is minimal and proximal.  Since $G$ is strongly amenable we can conclude that $R$ must be a singleton.
\end{proof}

Consider the set of all $x\in \beta G$ whose principal closed right ideal is minimal, that is,   
\[
    M=\{x\in \beta G : r(x) \text{ is a minimal closed right ideal}\}.
\]

\begin{prop}
\label{prop:finite orbit}
Let $G$ be a discrete group with $K(\beta G)$ the smallest two-sided ideal of $\beta G$, and fix $x \in M$. Let $J_x:= \bigcup_{g \in G} r(gx)$. Then $J_x$ and $M$ are both right ideals and left $G$-invariant. In particular,
\[
\overline{J_x}=\overline{M}=\overline{K(\beta G)}.
\]
As a result, $\{ r(gx) \ : \ g \in G \}$ is finite for some $x \in M$ if and only if $\beta G$ contains a unique minimal closed right ideal.
\end{prop}

\begin{proof}
Fix $x \in M$.
Note that $J_x$ is a right ideal and that it is left $G$-invariant.

Also, we always have $r(xy)\subset r(x)$ for $x,y \in \beta G$. Thus, if $x$ happens to be in $M$, by the minimality of $r(x)$ we must have that $r(xy)=r(x)$, and in particular, $r(xy)$ is minimal too. Thus, $M$ is a right ideal.

If $g \in G$ and $x \in M$ then $r(gx)$ is a minimal closed right ideal. Indeed, if $r(y) \subset r(gx)$ then $r(g^{-1} y) \subset r(x)$, so that $r(g^{-1} y) = r(x)$ by the minimality of $r(x)$, implying that  $r(y) = r(gx)$. Thus, $M$ is left $G$-invariant.

Now, note that $J_x \subset M \subset \overline{K(\beta G)}$. As $\overline{J_x}$ is a closed two sided ideal contained in the smallest one $\overline{K(\beta G)}$, the equality $\overline{J_x}=\overline{M}=\overline{K(\beta G)}$ must hold.

Finally, if $\{ r(gx) \ : \ g \in G \}$ is finite, then $J_x = \bigcup_{g \in G} r(gx)$ is closed,
so $J_x = \overline{M}=\overline{K(\beta G)}$. Thus, for every minimal closed right ideal $r(y)$ of $\beta G$, $y$ belongs to $r(gx)$ for some $g \in G$, and by minimality is equal to $r(gx)$. We conclude that $\beta G$ has finitely many minimal closed right ideals, so 
by dichotomy presented in Proposition \ref{pn:dichotomy}, it has a unique minimal closed right ideal.
\end{proof}

Given a nonempty subset $A$ of $\beta G$, we would like to measure how much it is invariant from the left, or, in other words, how many of its left translates it intersects. More precisely, for $\emptyset \neq A \subseteq \beta G$, we consider the sets of ``intersecting translates" of $A$, 
\[
\IT(A)=\{g \in G : A \cap gA \neq \emptyset\}.
\]
Note that $\IT(A)$ is symmetric (i.e. $g \in \IT(A)$ if and only if $g^{-1} \in \IT(A)$) and that $e \in \IT(A)$.
Also $\IT(gA)=g\IT(A)g^{-1}$ for every $g \in G$. 
It is simple to verify that if $A$ is left $G$-invariant (i.e. $gA \subset A$ for every $g \in G$) then $\IT(A) = G$,
and if $A$ is right $G$-invariant (i.e. $Ag \subset A$ for every $g \in G$) then $\IT(A)$ contains the center $Z(G)$ of $G$.

We are mainly interested in the case where $A=r(x)$ for some $x\in \beta G$. Note that if $x$ happens to be in $M$, then $r(x) \cap r(gx) \neq \emptyset$ if and only if $r(x)=r(gx)$. Thus, in this case, $\IT(r(x))$ is a {\em subgroup} of $G$ containing the center $Z(G)$. 

\begin{thm} \label{thm:IT r(x)}
Let $G$ be a discrete group. Then the following conditions are equivalent:
\begin{enumerate}
\item There is a unique minimal closed right ideal in $\beta G$.
\item $\IT(r(x))=G$ for every $x \in M$.
\item $[G:\IT(r(x))] < \infty$ for some $x \in M$.
\end{enumerate}
\end{thm}

\begin{proof} 
We start with $(1) \implies (2)$. Suppose that there is a unique minimal closed right ideal in $\beta G$. 
Then $r(x)=r(y)$ for every $x,y \in M$, and in particular $r(x)=r(gx)$ for every $x \in M$ and $g \in G$, 
so that $\IT(r(x))=G$ for all $x \in M$. 

The implication $(2) \implies (3)$ is obvious. 

We now show $(3) \implies (1)$.
Fix $x \in M$ and consider the map $r(gx) \mapsto g\IT(r(x))$.  One notes that this map is well defined and injective,
so that if $x \in M$ such that $[G:\IT(r(x))] < \infty$, then $\{ r(gx) \ : \ g \in G \}$ is finite.
By Proposition \ref{prop:finite orbit} we have that $r(y) = r(x)$ for all $y \in M$, \ie there is a unique minimal closed right ideal.
\end{proof}

\begin{cor}
If $G$ is not strongly amenable, then $\IT(r(x))$ is an infinite index subgroup of $G$ for every $x \in M$.
\end{cor}

We now prove that in virtually abelian groups, any two $\CS$ sets must intersect.

\begin{thm}\label{thm:virt_abelian_not_CS-dcmpsbl}
 If $G$ is virtually abelian, then $\beta G$ contains exactly one minimal closed right ideal. In particular, the intersection of any two $\CS$ subsets of $G$ is again a $\CS$ subset.
\end{thm}

\begin{proof}
Let $Q$ be the quasi-center of $G$, \ie
$$ Q = \{ g \in G \ : \ [G:C_G(g)] < \infty \} $$
where $C_G(g) = \{ h \in G \ : \ hg=gh \}$ is the centralizer of $g$.
(It is easy to see that $Q$ is a subgroup, which basically follows from the facts that
$C_G(g^{-1} h) \geq C_G(g) \cap C_G(h)$,
and that the intersection of finite index subgroups is itself of finite index.)

First, we claim that $Q \leq \IT(r(x))$ (for any $x \in \beta G$).

Indeed, if $g \in Q$, then we can write $\{ g^h \ : \ h \in G \} = \{ g^{h_1} , \ldots, g^{h_n} \}$ for some $h_1, \ldots, h_n \in G$.
(Here $g^h = h^{-1} g h$.)
For a net $(g_\alpha)_\alpha$ in $G$ converging to $x \in \beta G$, we have that 
$g g_\alpha = g_\alpha g^{g_\alpha} \in g_\alpha \cdot \{ g^{h_1} , \ldots, g^{h_n} \}$.
Since the latter set is finite, by passing to a subnet we may assume without loss of generality that 
$g g_\alpha = g_\alpha g^h$ for some $h \in \{ h_1, \ldots, h_n \}$ and all $\alpha$.
This implies that 
$$ r(gx) \ni gx = \lim g g_\alpha = \lim g_\alpha g^h = x g^h \in r(x) $$ 
so that $r(gx) \cap r(x) \neq \emptyset$.

This proves that $Q \leq \IT(r(x))$ for any $x \in \beta G$.

Specifically, if $G$ is virtually abelian, $[G:Q] < \infty$, so that $[G: \IT(r(x))] < \infty$.
Taking $x \in M$ and using Theorem \ref{thm:IT r(x)} we conclude that there is a unique minimal closed
right ideal.
\end{proof}

As mentioned, it is shown in \cite{KenRaumSal22} that if $G$ is a finitely generated non-virtually-nilpotent group,
then $G$ contains an infinite sequence of pairwise disjoint $\CS$ sets,
and therefore $\beta G$ contains infinitely many different minimal closed right ideals.
This is in stark contrast to the virtually-abelian case in Theorem \ref{thm:virt_abelian_not_CS-dcmpsbl}.
We have not been able to determine the remaining situation, \ie when $G$ is virtually nilpotent but not virtually abelian.
Even for specific concrete groups such as the Heisenberg group this is not clear.

\begin{ques}\label{ques:nilpotent}
Let $G$ be a nilpotent group which is not virtually abelian.

Does $\beta G$ necessarily contain a unique minimal closed right ideal? 
\end{ques}

The finite index that arises for considering virtually nilpotent groups is not critical, as can be seen from the following.

\begin{prop}
\label{prop:finite index}
Let $H$ be a finite index subgroup of $G$.  Then $\beta H$ has a unique minimal closed right ideal if and only if 
$\beta G$ has a unique minimal closed right ideal.
\end{prop}

\begin{proof}
We identify $\beta H$ with $\overline{H} \subset \beta G$ and denote $r_H(x):=\overline{x\overline{H}}$, for $x \in \overline{H}$. We begin with a simple observation. Let $R$ and $L$ be some sets of representatives for the right and left cosets of $H$ in $G$, respectively. Then we have that $\bigsqcup_{p \in L} p H = G = \bigsqcup_{p \in R} Hp$, and ince $R$ and $L$ are finite we can conclude that
\[
\bigsqcup_{p \in L} p \overline{H}=\beta G=\bigsqcup_{q \in R}  \overline{H}q.
\]

A key observation we will use is that 
$$ r(x) = \bigsqcup_{q \in R} r_H(x) q . $$
This can be verified as the above union is a closed right ideals contained in $r(x)$ and containing $x$.

{\bf ``If'' direction:}
If $\beta H$ has two distinct minimal closed right ideals $r_H(x)$ and $r_H(y)$ for some $x,y \in \overline{H}$, then 
\[
r(x)=\bigsqcup_{q \in R}  r_H(x)q \quad \text{and} \quad  r(y)=\bigsqcup_{q \in R}  r_H(y)q.
\]
Since the union $\beta G=\bigsqcup_{q \in R}  \overline{H}q$ is disjoint, and since $r_H(x)q$ and $r_H(y)q$ are supported on $\overline{H}q$ for each $q\in R$, the closed right ideals $r(x)$ and $r(y)$ must be distinct (in fact, they are also minimal, but this is not required for our argument). In particular, $\beta G$ contains two distinct minimal closed right ideals.

{\bf ``Only if'' direction:}
Assume now that $\beta H$ has a unique minimal closed right ideal. We find some $x \in M$ such that $[G:\IT(r(x))] < \infty$. This, by Theorem \ref{thm:IT r(x)}, implies that $\beta G$ contains a unique minimal closed right ideal as well. 

Fix some $y \in M$. Since $\beta G=\bigsqcup_{p \in L} p \overline{H}$, there is some $p \in L$ such that $x:=p^{-1}y \in \overline{H}$. Note, however, that $x$ is also in $M$. Since $r(x)=\bigsqcup_{q \in R}  r_H(x)q$ we must have that $r_H(x)$ is  minimal too (as a closed right ideal of $\overline{H}$). Thus, $r_H(hx)=r_H(x)$ for every $h \in H$, and so 
$$
r(hx)=\bigsqcup_{q \in R}  r_H(hx)q=\bigsqcup_{q \in R}  r_H(x)q=r(x)$$ 
for every $h \in H$. We conclude that $H \leq \IT(r(x))$, and so $[G:\IT(r(x))] \leq [G:H] < \infty$.
\end{proof}

\section{$\CS$ sets in $\Z$}

In this section, we provide a characterization of completely syndetic sets of integers. We denote by $\N$ the set of natural numbers, including $0$. 
Let $\Omega$ denote the collection of all subsets of the natural numbers that are neither finite nor co-finite and that contain $0$, that is,
$$ \Omega := \{ A \subset \N \ : \ \sup A = \sup ( \N \setminus A) = \infty \ , \ 0 \in A \}. $$
In addition, set 
$$ \Psi  :  = \{ (\alpha_n, \beta_n)_{n=1}^\infty \ : \ \alpha_n , \beta_n \in \N \setminus \{0\} \}. $$
The following proposition shows that sets in $\Omega$ are in a 1--1 correspondence with sequences in $\Psi$ that preserves the information regarding the structure of sets in $\Omega$.

\begin{prop} \label{prop:correspondence}
For each $A \in \Omega$, let $\alpha_n(A)$ be the length of the $n$-th block of consecutive elements in $A$, and let $\beta_n(A)$ be the length of the $n$-th block of consecutive elements in $\mathbb{N} \setminus A$. Then the map
\[
A \mapsto (\alpha_n(A),
\beta_n(A))_{n=1}^\infty
\]
is a well-defined bijection from $\Omega$ onto $\Psi$.
\end{prop}

\begin{proof}
Given $A \in \Omega$, let $\alpha_n = \alpha_n(A)$ 
denote the length of the $n$-th block of consecutive elements in $A$,
and $\beta_n = \beta_n(A)$ 
denote the length of the $n$-th block of consecutive elements in $\mathbb{N} \setminus A$. 
More precisely, define $(\alpha_n,\beta_n)_{n=1}^\infty$ inductively, as follows: set
$$
\alpha_1:= \min \{ z \in \N \ : \ z+1 \not\in A \} +1,
$$
and define inductively for $n \geq 1$
\begin{align*}
\beta_{n} & = \min \left\{ z \in \N \ : \ z \geq \alpha_n+\sum_{i=1}^{n-1}(\alpha_i+\beta_i)   \ , \ z+1 \in A \right\} - \left(\alpha_n+\sum_{i=1}^{n-1}(\alpha_i+\beta_i)\right)+1
\\
\alpha_{n+1} & = \min \left\{ z \in \N \ : \ z \geq \sum_{i=1}^{n}(\alpha_i+\beta_i)  \ , \ z+1 \not\in A \right\} - \left(\sum_{i=1}^{n}(\alpha_i+\beta_i)\right)+1
\end{align*}
These are well defined under the assumptions of sets in $\Omega$, and $(\alpha_n, \beta_n)_{n=1}^\infty \in \Psi$.

Conversely, if $(\alpha_n, \beta_n)_{n=1}^\infty \in \Psi$, we define  $A \subset \N$ by setting
$\ell_1 = 0$ and inductively for $n \geq 1$,
$$ r_n = \ell_n + \alpha_n \qquad  \ell_{n+1} = r_n + \beta_n $$
and then 
$$ A = \bigsqcup_{n=1}^\infty \big( [\ell_n , r_n) \cap \N \big) . $$

It is simple to verify that these maps are inverses of each other.
\end{proof}

\begin{remark}\label{remark:syndetic=>beta_is_bdd}
Note that if $A^+ := A \cap \N \in \Omega$ for some syndetic set $A \subseteq \Z$ (\ie, $A$ is $1$-syndetic), then $\sup_n \beta_n(A^+) < \infty$. Similarly, if $A^- := (-A) \cap \N \in \Omega$, then $\sup_n \beta_n(A^-) < \infty$ as well. Conversely, if both $A^+$ and $A^-$ belong to $\Omega$, and both $\sup_n \beta_n(A^\pm)$ are finite, then A must be syndetic.

This follows directly from the fact that syndetic sets have bounded gaps --- equivalently, their complements are not thick (see the discussion preceding Corollary~\ref{cor:syndethick_yet_non_CS}). In fact, this is the origin of the term {\em syndetic}.

We can also give a direct proof: If $A$ is syndetic, there exists a finite set $F_1 \subset \Z$ such that for every $z \in \Z$, some $f \in F_1$ satisfies $z + f \in A$. Now suppose, toward a contradiction, that $\sup_n \beta_n = \infty$. Then for some large even $\beta_n > \max\{2|f| : f \in F_1\}$, there exists $z \in \N$ such that
\[
(z + F_1) \cap A \subset \{z - \tfrac{1}{2} \beta_n + 1, \ldots, z + \tfrac{1}{2} \beta_n - 1\} \cap A = \emptyset,
\]
contradicting the syndeticity of $A$. Therefore, $\sup_n \beta_n(A^+) < \infty$.   

Conversely, if both $A^+$ and $A^-$ are in $\Omega$ and both $\sup_n \beta_n(A^+)$ and $\sup_n\beta_n(A^-)$ are finite, choose some 
$b \geq \sup_n \{ \beta_n(A^+) , \beta_n(A^-) \}$.

Fix $z \in \N$. If $\{z, z+1, \ldots, z+b\} \subset \N \setminus A^+$, then there exists some $k$ such that $\beta_k(A^+) \geq b+1$, contradicting the choice of $b$. Thus we can conclude that for every $z \in \N$, there exists $f \in {0, 1, \ldots, b}$ such that $z + f \in A^+$.

A similar argument applied to $A^-$ shows that for every $z \in \N$, there exists $f \in \{-b,-b+1, \ldots, 0\}$ such that $-z + f \in A^-$.

So, for every $z \in \Z$, there exists $f \in F_1 := \{-b, -b+1, \ldots, b\}$ such that $z + f \in A$, and so $A$ is syndetic.
\end{remark}

We now turn to the classification of $\CS$ sets in $\Z$. In order to do so, we need to define a special class of sets we call ``uniform-sub-Szemer\'{e}di''.  

\begin{dfn}
An
\define{sub-arithmetic progression} of difference $D$ and length $L$ is a finite sequence of integers 
$p_1 < p_2 < \cdots < p_L$ such that $p_{n+1} - p_n \leq D$ for all $1 \leq n < L$.
\end{dfn}

\begin{dfn} \label{dfn:USS}
We say that $A \subset \Z$ is a \define{$D$-uniform-sub-Szemer\'{e}di} set, or $D$-$\USS$ set, 
if $A$ contains arbitrarily long sub-arithmetic progressions of difference $D$.

We say that $A$ is a \define{uniform-sub-Szemer\'{e}di} set, or $\USS$ set, if $A$ is $D$-$\USS$ for some $D$.
\end{dfn}

Note that a set $A$ is $\USS$ if and only if at least one of $A \cap \N$ or $(-A) \cap \N$ is $\USS$. Therefore, when discussing the $\USS$ property, we freely move between subsets of $\N$ and subsets of $\Z$.

The name {\em sub-Szemer\'{e}di set} comes from the following.
A set $A \subset \N$ is called a \define{Szemer\'{e}di} set if it contains arbitrarily long arithmetic progressions.
(Szemer\'{e}di in~\cite{szemeredi1969sets} proved the conjecture of Erd\fH{o}s and T\'{u}ran proposed in~\cite{erdos1936some} that any set with positive density is a Szemer\'{e}di set.)
The prefix {\em sub} is added because we only bound the difference from above, and do not insist on a constant distance.
The adjective {\em uniform} means that the distance is uniformly bounded for all of the sub-arithmetic progressions.

\begin{lem} \label{lem:CompNotUSS}
Let $A \in \Omega$ with the corresponding sequence $(\alpha_n, \beta_n)_{n=1}^\infty \in \Psi$. 

Then,
$\mathbb{N} \setminus A$ is not $\USS$ if and only if $\sup_n \beta_n < \infty$ and for every $D$ there exists some 
integer $L(D)>0$ such that for every $m \in \mathbb{N}$ there exists some $0 \leq j < L(D)$ such that $\alpha_{m+j} \geq D$.
\end{lem}

\begin{proof}
The set $\N \setminus A$ is $D$-$\USS$ if and only if for every $L$ there is either an $n$ such that $\alpha_{n+j} < D$ for all $0 \leq j < L$, or an $n$ such that $\beta_n \geq L$. Thus, $\N \setminus A$ is not $\USS$ if and only if for every $D$ there is some $L$ such that for all $n$  one can find $0 \leq j < L$ with $\alpha_{n+j} \geq D$, and for all $n$ we have $\beta_n < L$. This is equivalent to the statement that for every $D$ there exists $L$ such that $\sup_n \beta_n < L$ and for all $n$ there exists $0 \leq j < L$ such that $\alpha_{n+j} \geq D$.
\end{proof}

\begin{thm}\label{thm:CSinZ}
Let $A \subset \Z$. Then $A$ is $\CS$ if and only if $\Z \setminus A$ is not $\USS$. 
\end{thm}

\begin{proof}
If $A$ is bounded above or below --- that is, if either $A^+:=A \cap \N$ or $A^-:=(-A) \cap \N$ is finite --- then $A$ is not $\CS$ (not even syndetic), and $\Z \setminus A$ is $1$-$\USS$.

On the other hand, if $\Z \setminus A$ is bounded both above {\em and} below --- that is, if $A$ is cofinite in $\Z$ --- then $A$ is $\CS$, and $\Z \setminus A$ is not $\USS$.

Since both $\CS$ and $\USS$ are invariant under translations, we may assume without loss of generality that $0 \in A$.

We may thus reduce to the case where either both $A^+$ and $A^-$ lie in $\Omega$, or one lies in $\Omega$ and the other is cofinite in $\N$.

Since these properties are also invariant under reflection (\ie under $A \mapsto -A$), we may further assume that $A^+ \in \Omega$ and that $A^-$ is either in $\Omega$ or cofinite in $\N$.

Finally, since both $\CS$ and $\USS$ are invariant under translations, we may assume that $A^+ \in \Omega$ and that $A^-$ is either in $\Omega$ or equals to all of $\mathbb N$.

{\bf ``Only if'' direction:}
Assume that $A \subset \Z$ is $\CS$.

Consider $A^+\in \Omega$, and let $(\alpha_n,\beta_n)_n \in \Psi$ be the corresponding sequence constructed in Proposition \ref{prop:correspondence}.
Since $A$ is $\CS$, it is in particular syndetic, so by Remark~\ref{remark:syndetic=>beta_is_bdd}, we have $\sup_n \beta_n < \infty$.

Now, we write $A^+ = \bigsqcup_{k=1}^\infty ([\ell_k, r_k) \cap \N)$
where $\ell_1=0$, $r_k = \ell_k + \alpha_k$, and $\ell_{k+1} = r_k + \beta_k$.

Let $D \in \N$. Since $A$ is $D$-syndetic, there is a finite set $F_{D}$ such that for every $z \in \Z$
there exists $f \in F_{D}$ with $\{ z+1+f, \ldots, z+D+f \} \subset A$.
Define $L(D) = \max \{ 2 |f| +1 \ : \ f \in F_{D} \}$, and set $J=\tfrac{L(D)-1}{2}$.

Let $m \in \N$.  
Consider $z = \ell_{m+J} \in A^+$. We know that $\{ z+1+f, \ldots , z+D+f \} \subset A^+$ for some $f \in F_{D}$.
Since $z+1+f \in A^+$, 
there must exist $|\rho| \leq |f| \leq J$ such that $z+1+f \in [\ell_{m+J+\rho} , r_{m+J+\rho} )$.
Since $\{ z+1+f, \ldots , z+D+f \} \subset A^+$, it follows that $r_{m+J+\rho} \geq 1 + z+D+f$ and $\ell_{m+J+\rho} \leq z+1+f$, 
which implies that 
$$ \alpha_{m+J+\rho} = r_{m+J+\rho}-\ell_{m+J+\rho} \geq D . $$
Thus, for each $m$ we have found some $0 \leq j: = J+\rho < L(D)$ for which $\alpha_{m+j} \geq D$.

By Lemma \ref{lem:CompNotUSS}, this shows that $\N \setminus A^+$ is not $\USS$.

If $A^-=\N$, then $\Z \setminus A=\N \setminus A^+$, and hence not $\USS$.

Otherwise, if $A^-=(-A) \cap \N \in \Omega$, then by applying the previous argument to $-A$ (which is $\CS$ as well), we conclude that $\N\setminus A^-$ is not $\USS$. It follows that
\[
\Z \setminus A = (\N \setminus A^+) \cup -(\N \setminus A^-)
\]
is also not $\USS$.

{\bf ``If'' direction:}
We now prove the other direction. Assume that $\Z \setminus A$ is not $\USS$, that is, both $\N \setminus A^+$ and $\N \setminus A^-$ are not $\USS$. We first assume that $A^-$ lies in $\Omega$, in addition to  $A^+$. 

We prove this by induction on $n$ that $A$ is $n$-syndetic.

Let $(\alpha_n^{\pm},\beta_n^{\pm})_n , \in \Psi$ be the sequences corresponding, by Proposition \ref{prop:correspondence}, to the sets $A^\pm$, respectively.
By Lemma \ref{lem:CompNotUSS}, $b:=\sup_n\{\beta_n^+,\beta_n^-\}<\infty$, so by Remark~\ref{remark:syndetic=>beta_is_bdd}, $A$ is $1$-syndetic, completing the base of the induction.

Now, assume that $A$ is $n$-syndetic.
We prove that $A$ is $(n+1)$-syndetic.

Our assumption is that 
there is some finite set $F_n$ so that for any set $\hat S$ of size $|\hat S|=n$
there exists some $f \in F_n$ for which $\hat S + f \subset A$.

Since $\N \setminus A^\pm$ are both not $\USS$,
for every $D$ there exists $L(D)$ such that for all $m$ there are some $i,j \leq L(D)$ with $\alpha_{m+j}^+ \geq D$
and $\alpha_{m+i}^- \geq D$.

Let $M_n = \max \{ |f| \ : \ f \in F_n \}$ and define
$$ F_{n+1} = \{ - (b+2M_n) (L(2M_n+1)+3) , \ldots, (b+2M_n) (L(2 M_n+1)+3) \} . $$
(Recall that $b = \sup_n \{ \beta_n^+ , \beta_n^- \}$.)

Let $S \subset \Z$ be of size $|S| = n+1$. We want to find some $f \in F_{n+1}$ such that $S+f \subset A$.

Since $F_{n+1} = -F_{n+1}$, and since $S+f \subset A$ if and only if $-S-f \subset -A$,
and since $A$ is $n$-syndetic if and only if $-A$ is $n$-syndetic, we may assume without 
loss of generality that $u : = \max S > 0$.

Set $\hat S = S \setminus \{u\}$, so that $|\hat S | = |\hat S + M_n| = n$.
Let $f \in F_n$ be such that $\hat S + M_n + f \subset A$.
Set $m := M_n+f$ and note that $0 \leq m \leq 2 M_n$ so that $m \in F_{n+1}$.

If $u+m \in A$ then $S+m \subset A$ and we are done because $m \in F_{n+1}$.

Otherwise, if $u+m \not\in A$, then $u+m \in \N \setminus A^+$. 
Writing $A^+ = \bigsqcup_{k=1}^\infty ([\ell_k, r_k) \cap \N)$
where $\ell_1=0$, $r_k = \ell_k + \alpha_k^+$, and $\ell_{k+1} = r_k + \beta_k^+$,
we find that $u+m \in [r_k,\ell_{k+1})$ for some $k \geq 1$.

Set $D := 2M_n+1$.
Now, because $\N \setminus A^+$ is not $\USS$, there exists a minimal $0 \leq j < L(D)$
such that $\alpha_{k+1+j}^+ \geq D$. 
This implies that 
\[
\begin{split}
\ell_{k+1+j} - (u+m) 
&\leq  \ell_{k+1+j} - r_k \\
&\leq \beta_k^+ + (\alpha_{k+1}^+ +\beta_{k+1}^+) + \cdots + (\alpha_{k+j}^+ +\beta_{k+j}^+)\\ 
&\leq (j+1)b+j(D-1)\\
&\leq (j+1)(2M_n+b)  .
\end{split}
\]  
Thus, $\ell_{k+1+j} + M_n=u+m'$ for some $0 \leq m' \leq m + (2M_n+b)(j+1) + M_n \leq (2M_n+b) (L(D)+2)$. As $\alpha_{k+1+i}^+\geq D$ we have that $u+m' = \ell_{k+1+j} + M_n$ is in $A$.  

Now, since $|\hat S + m'|= n$, there is some $f' \in F_n$ such that $\hat S + m'+ f' \subset A$.
However, since $|f'| \leq M_n$, we have that $$\ell_{k+1+j}=u+m'-M_n \leq u+m'+f' \leq \ell_{k+1+j} + 2M_n < r_{k+1+j},$$
implying that $u+m'+f' \in A$ as well.

We conclude that, in either case, either $S+m \subset A$, or that $S+m' + f' \subset A$.
Since $0 \leq m \leq 2 M_n$ and $|m'+f'| \leq (2M_n+b) (L(2M_n+1)+3)$, it follows that both $m \in F_{n+1}$ and $m'+f' \in F_{n+1}$.
This completes the induction step, proving that $A$ is $\CS$.

Finally, the case where $A^- = \N$ can be treated in the same way: simply repeat the previous argument using $b := \sup_n \{\beta_n^+\}$ in place of $\sup_n \{ \beta_n^+, \beta_n^- \}$, and omit any mention of $\alpha_n^-$ or $\beta_n^-$. Note also that if $u := \max S \leq 0$, then the conclusion follows immediately, and there is no need to appeal to the reflection symmetry $A \mapsto -A$ as we did earlier.
\end{proof}
Recall that $K(\beta \Z)$ denotes the unique smallest two-sided ideal in $\beta \Z$.
\begin{cor}\label{cor:ol_A_contains_K(beta Z)}
    Let $A \subset \Z$. Then $\overline{A} \supset K(\beta \Z)$ if and only if $\Z \setminus A$ is not $\USS$.
\end{cor}
\begin{proof} 
    By Proposition \ref{prop:CS and right ideals}, $A$ is $\CS$ if and only if $\overline{A}$ contains a right ideal of $\beta \Z$. So by Theorem \ref{thm:virt_abelian_not_CS-dcmpsbl}, $A$ is $\CS$ if and only if $\overline{A} \supset K(\beta \Z)$. Thus, by Theorem \ref{thm:CSinZ}, $\overline{A} \supset K(\beta \Z)$ if and only if $\Z \setminus A$ is not $\USS$ 
\end{proof}

We close this section with the a construction of a thick and syndetic set that is not $\CS$.
Recall that $A \subset G$ is \define{left (resp. right) 
thick} if for any finite set $S \subset G$ there exists some $g \in G$ such that $gS \subset A$ (resp. $Sg \subset A$). It is easy to observe that a set $A\subset G$ is syndetic ({\em i.e.}, 1-syndetic) if and only if its complement $G\setminus A$ is not right thick: this simply follows from the fact that $SA=G$ if and only if there is no $g \in G$ such that $S^{-1}g \subset G\setminus A$, for every $S,A \subset G$, where $S^{-1}=\{s^{-1}: s\in S\}$; see \cite[Remark 4.46]{HindmanStrauss}.

Note also that $\CS$ sets are automatically left thick: in fact, if $A \subset G$ is $n$-syndetic and $S\subset G$ with $|S| \leq n$ then, by definition, there is $g \in G$ such that $gS \subseteq A$. Also, obviously, $\CS$ sets are syndetic. It is therefore natural to ask whether syndetic and thick sets are automatically $\CS$.

As in syndeticity, if we say that a set is just \define{thick}, we mean that it is {\em left} thick. In any event, for abelian groups, right and left thickness obviously coincide. For subsets of integers, thick sets are clearly those that contain arbitrarily large ``chunks'' of consecutive integers. Recall also that
syndetic subsets of integers are those that have bounded gaps.

\begin{cor}\label{cor:syndethick_yet_non_CS}
There are thick and syndetic sets $A \subset \mathbb Z$ that are not $\CS$.
\end{cor}

\begin{proof}
Set $B=\{2^n+2k \ | \  n \in \mathbb N, ~0 \leq k \leq n-1 \}$ and let $A\subset \mathbb Z$ be the complement of $B \cup (-B) \subset \Z$. 
We show that $A$ is thick and syndetic but not $\CS$. 

By Theorem \ref{thm:CSinZ} $A$ is not $\CS$, as $B$ contains arbitrarily long arithmetic progressions of difference exactly $2$ (so specifically $B$ is $2$-$\USS$). 

The set $B \cup (-B)$ is obviously not syndetic since it does not have bounded gaps. 
Thus, $A$ is thick. 

In addition, since $B$ is contained in the even natural numbers, 
there is obviously no $z \in \mathbb Z$ such that $z + \{1,2\}$ is in $B \cup (-B)$. Thus, $B \cup (-B)$ is not thick. Hence,
$A$ is syndetic.
\end{proof}

\begin{remark}
One can in fact conclude that $A$ is not $\CS$ by showing that it is not $2$-syndetic
(although, as mentioned, it is thick and $1$-syndetic).

Indeed, let $F \subset \Z$ be any finite set, and let $r = \max \{ |f| \ : \ f \in F \}$. Then, choose $b = 2^{2r} + r$ and $a = b-1$.
Note that if $f \in F$ then $2^{2r} \leq b+f \leq 2^{2r}+2r$.
Therefore, if $b+f$ is even, then $b+f \in B$ and so $b+f \not\in A$.
Otherwise, if $b+f$ is odd, then $b+f > 2^{2r}$, so $2^{2r} \leq a+f \leq 2^{2r}+2r$ and $a+f$ is even, implying that $a+f \in B$ so $a+f \not\in A$.

We have shown that for any finite $F \subset \Z$ there exists a set $\{a,b\}$
such that $\{a,b\} + f \not\subset A$ for all $f \in F$.
That is, $A$ is not $2$-syndetic.
\end{remark}

\section{$\CS$ set with small density}

Since $\CS$ sets are in particular syndetic, they necessarily have positive density (since finitely many of their left translates cover the entire group). In this section, we show that positivity is the only constraint: $\CS$ sets can have an arbitrarily small positive density.

If $G$ is a group generated by a finite set $S$, the \define{word metric} $\rho_{(G,S)}$ on $G$ with respect to $S$ is the metric on $G$ associated with the Cayley graph of $G$ with respect to $S$. More formally, given $g,h \in G$, we define
\[
\rho_{(G,S)}(g,h)=\min\left\{  n \in \N ~:~ g^{-1}h=s_1\cdots s_n \text{ for some }   s_1,\dots,s_n \in S \cup S^{-1}\right\}.
\]
With respect to this metric, let $B_r$ denote the closed ball of radius $r$ around the identity element $e$, \ie the set of all $g \in G$ with $\rho_{(G,S)}(g,e) \leq r$.

The \define{density} $d_{(G,S)}$ of a subset $A\subset G$ with respect to $S$ is defined by

\[ 
d_{(G,S)}(A)=\limsup_{r \to \infty} \frac{ |A \cap B_r |}{|B_r| }.  
\]

We first focus on the group of integers $\Z$. In this context, when we say {\em density} we always mean the density with respect to the generating set $\{\pm 1 \}$ (or equivalently, with $\{ 1 \}$). That is, the density of a set $A \subset \Z$ is
\[
\limsup_{r \to \infty} \frac{ |A \cap [-r,r] |}{2r+1 }. 
\]

We use the characterization of Theorem \ref{thm:CSinZ} to produce a $\CS$ set in $\Z$ that has an arbitrarily small positive density.

First, we need an easy auxiliary lemma:

\begin{lem} \label{lem:recurrence}
Fix $T>1$ and $s \geq 0$.
Consider the recursion $r_{n+1} = T r_n + n+1+s$, with $r_0=0$.

Then, $r_n = a (T^n-1) + bn$ for
$$  a = \frac{s(T-1) + T}{(T-1)^2}
\AND b = - \frac{1}{T-1} . $$
\end{lem}

\begin{proof}
Since $(r_n)_{n=0}^\infty$ is a linear recurrence with constant coefficients then there exists some constants $a,b,c$ such that $r_n = a T^n + b n + c$ for all $n$.
This gives rise to the equations: $a+c=r_0=0$ and 
$$ a (T-1) + b = r_1 = 1+s \qquad a (T^2-1) + 2 b = r_2 = T(1+s) + 2 +s $$
which leads to the solution.
\end{proof}

\begin{thm} \label{thm:small CS in Z}
For any $\eps>0$ there exists a $\CS$ set $A \subset \Z$ with density less than $\eps$.
\end{thm}

\begin{proof}
Let $K>0$ and $M>1$ be positive integers.
Set $\beta_n = d$ for all $n \geq 1$.
Define inductively: $\gamma_{1,1} = 1$, $r_1=1$, $r_{n+1} = M r_n +1$, $\gamma_{n+1,r_{n+1}} = n+1$,
and for $1 \leq k \leq M r_n$ define 
$$ \gamma_{n+1,k} = \gamma_{n , \ell(k)} \qquad \ell(k) = (k-1 \pmod{r_n}) + 1 $$

Now, define $(\alpha_n)_n$ by concatenating the ``rows'' of the triangular array $\gamma$.
That is, set $R_0=0$ and $R_m = \sum_{k=1}^m r_k$. Then, for $n \geq 1$, 
set $B(n) = m$ for any $m$ such that $R_{m} < n \leq R_{m+1}$.
Then, set $C(n) := n-R_{B(n)}$,
so that $1 \leq C(n) \leq r_{m+1}$ for $m=B(n)$.
Define 
$$ \alpha_n = \gamma_{B(n)+1 ,  C(n) } . $$

\eg if $M=2$ we have that the first values of $\gamma_{n,k}$ are:
$$ 
\begin{matrix}
1 & &   &   &   &   &  &  \\
112 &  &  &   &   &  &  \\
112 & 112 & 3 &  &  &  &  \\
112 & 112 & 3 & 112 & 112 & 3 & 4 
\end{matrix}
$$
and the sequence $\alpha$ starts out:
$$ 1\ 112\ 1121123\ 112112311211234\ \ldots $$

Let $A^+ \in \Omega$ be defined by the sequence $(\alpha_n,\beta_n)_n \in \Psi$.

For any $D$ and $n > R_D$ we have that $\alpha_{n+k} = \gamma_{B(n+k)+1, C(n+k)}=\gamma_{B(n)+1, C(n)+k} > D$ for some $k \leq r_{D+1}$.  This implies that $\N \setminus A^+$ is not $\USS$.
Hence, by taking $A = A^+ \cup -A^+$ we get that $A$ is $\CS$ by Theorem \ref{thm:CSinZ}.

We now show that $A$ has small density.

The recursion $r_{n+1} = M r_n + 1$ with $r_1=1$ is solved by $r_n = \frac{M^n-1}{M-1}$, so that
\begin{align*}
R_m & = \frac{1}{M-1} \sum_{k=1}^m (M^k - 1) = - \frac{m}{M-1} + \frac{M(M^m-1)}{(M-1)^2} .
\end{align*}
Also, denote $\Gamma_n := \sum_{k=1}^{r_n} \gamma_{n,k}$ (the sum over the $n$-th ``row'' of $\gamma$).
We have the recurrence $\Gamma_{n+1} = M \cdot \Gamma_n + n+1$, with $\Gamma_1=1$, 
so this recurrence is solved by
$$ \Gamma_n = \frac{ M^{n+1} - M  - (M-1) n }{ (M-1)^2} $$
as in Lemma \ref{lem:recurrence}.
Summing this over $n$ we have
\begin{align*}
\sum_{k=1}^{R_m} \alpha_k & = \sum_{n=1}^m \Gamma_n \leq  \frac{1}{(M-1)^2} \sum_{n=1}^m M^{n+1} 
\leq \frac{M^{m+2} }{ (M-1)^3 } .
\end{align*}
We conclude that as $m \to \infty$,
$$ 1 \leq  \frac{1}{R_m} \sum_{n=1}^{R_m} \alpha_n \leq \frac{ M^{m+2} }{ M (M-1) (M^{m} - 1) - m (M-1)^2 } \to 1 . $$
Hence, the density of $A$ is given by
$$ \lim_{m \to \infty} \frac{ \sum_{n=1}^{R_m} \alpha_n }{ \sum_{n=1}^{R_m} \left(\alpha_n+\beta_n\right) }
= \lim_{m \to \infty} \left( 1 + K \cdot \frac{R_m}{\sum_{n=1}^{R_m} \alpha_n } \right)^{-1} = \frac{1}{K+1} .
$$
As $K$ is arbitrary, this can be made as small as desired.
\end{proof}

\begin{remark}\label{rem:Z^d}
 We can directly generalize Theorem~\ref{thm:small CS in Z} to $\mathbb{Z}^d$ for any $0<d \in \N$. Let $\varepsilon>0$ and let $A \subset \mathbb{Z}$ be a $\CS$ set with density less than $\sqrt[d]{\varepsilon}$. Then it is simple to check that $A^d \subset \Z^d$ is $\CS$ as well, and its density is less than $\varepsilon$.
\end{remark}

In Theorem \ref{thm:small CS in any group} we show that Remark~\ref{rem:Z^d} can be generalized to any finitely generated discrete infinite group. We start with a few lemmata.

\begin{lem} \label{lem:CS from quotient}
Let $\vphi : G \to H$ be a surjective homomorphism.

If $A \subset H$ is $\CS$ in $H$, then $\vphi^{-1}(A) \subset G$ is $\CS$ in $G$.
\end{lem}

\begin{proof}
For each $n$ let $F_n \subset H$ be a finite set such that for any $S \subset H$ of size at most $n$
there exists $f \in F_n$ such that $S \subset f A$.

For each $n$ and $f \in F_n$ choose some $\tilde f \in G$ such that $\vphi(\tilde f) = f$.
Let $\tilde F_n = \{ \tilde f \ : \ f \in F_n \}$.  So $\tilde F_n$ is a finite set.

Let $T \subset G$ be a subset of size at most $n$.
Consider $\vphi(T) \subset H$.  Since $|\vphi(T)| \leq n$, there exists $f \in F_n$ such that $\vphi(T) \subset f A$.
Thus, $\vphi( \tilde f^{-1} T ) \subset A$, which implies that $T \subset \tilde f \vphi^{-1}(A)$.
\end{proof}

Recall that we say that a group $G$ with a finite generating set $S$ is said to have \define{polynomial growth} if there exists some polynomial $p(\cdot)$ such that for every $n$ we have that $|\{g \in G \colon \rho_{(G,S)}(g,e) \leq n  \}| \leq p(n)$, where $\rho_{(G,S)}(g,e)$ is the distance of $g$ to the identity element $e$ in the word metric $\rho_{(G,S)}$ with respect to $S$. 
It is well known (and easy to prove) that the notion of polynomial growth does not depend on the choice of specific generating set $S$.
A famous theorem by Gromov states that a finitely generated group has polynomial growth if and only if it is virtually nilpotent \cite{Gromov}.

\begin{lem}
\label{lem:density}
Let $G$ be a finitely generated group of polynomial growth,
and let $\vphi : G \to \Z$ be a surjective group homomorphism. 

Fix a finite generating set $S \subset G$.

Then, there is $C>0$ such that $d_{(G,S)}(\vphi^{-1}(A)) \leq C d_{(\Z,\{\pm 1\})}(A)$ for any $A \subset \Z$.
\end{lem}

\begin{proof}
Let $K :=  \mathsf{Ker}  \vphi$ and fix an $a \in G$ such that $\vphi(a) = 1$. Since $G/K \cong \Z$, every $g \in G$ can be written uniquely as $g = k a^z$ for some $k \in K$, $z \in \Z$.

Consider some generating set of $G$.
Let $B_r$ denote the ball of radius $r$ around the identity element in $G$ in the corresponding metric.
We use the notation $|g|$ to denote the distance of $g$ to the identity element $e$ of $G$.
So the distance between $g,h$ is given by $|g^{-1} h|$.
Set $M = |a|$.

Now, if $g \in  B_r$ then $g a^z \in B_{r+M|z|}$.
So, for any fixed $w \in \Z$, the map $(ka^w,z) \mapsto k a^{w+z}$, from $K a^w \times \Z$ to $G$, is injective, and maps 
$(Ka^w \cap B_r) \times ([-r,r] \cap \Z)$ into $B_{(M+1)r}$. This implies that
$|B_{(M+1)r}| \geq (2r+1) |K a^w \cap B_r |$ for any $w \in \Z$.

Also, if $k a^z \in B_r$, then since $\vphi$ is Lipschitz (as any group homomorphism is a Lipschitz map between any two Cayley graphs), 
we have that $|z| = |\vphi(ka^z)| \leq Lr$, where $L$ is the Lipschitz constant of $\vphi$.
Specifically, if $k a^z \in \vphi^{-1} (A) \cap B_r$ then $z \in A \cap [-Lr,Lr]$.

Thus, we have that
\begin{align*}
\frac{ |\vphi^{-1} (A) \cap B_r | }{ |B_r| } & = \sum_{w \in A \cap [-Lr,Lr] } \frac{ |K a^w \cap B_r| }{ |B_r| }
\leq \frac{ | A \cap [-Lr,Lr] | }{2r+1} \cdot \frac{ |B_{(M+1)r}| }{ |B_r | } .
\end{align*}
Since $G$ has polynomial growth, there is a some $C>0$, dependent only on $M$, such that $\limsup_{r \to \infty} \frac{|B_{(M+1)r}|}{ |B_r|} \leq C < \infty$
(this is well known and follows from \cite{Bass, Guivarch}), so 
$d_G(\vphi^{-1}(A)) \leq C L d_{\Z}(A)$.
\end{proof}

\begin{thm}\label{thm:small CS in any group}
Let $G$ be an infinite finitely generated group,
and fix some finite generating set $S$ of $G$.

For any $\eps>0$ there exists a $\CS$ set $A \subset G$ with density less than $\eps$.
\end{thm}

\begin{proof}
If $G$ is finitely generated and not virtually nilpotent, then by \cite{FriTamVa18} it is not strongly amenable. Thus, it follows from \cite{KenRaumSal22} that there exists a sequence
$(A_n)_n$ of pairwise disjoint $\CS$ sets, some of which must have arbitrarily small density.

Now, assume that $[G:H] < \infty$, and that $A \subset H$ is $\CS$ (in $H$) with density less than $\eps$.
Let $G = \bigsqcup_{r \in R} Hr$ for some set of representatives $R$ of the right cosets (in particular, $|R| = [G:H]$), and define $B = \bigsqcup_{r \in R} Ar$.
It is simple to check that the density of $B$ in $G$ is less than $C \eps$, for some constant $C>0$ 
(depending on the specific generating sets of $G$ and $H$ with respect to which the densities are taken, as well as on the index $[G:H]$).

We claim that $B$ is $\CS$ in $G$.  Indeed, for every $n$ there is a finite set $F_n \subset H$ such that for any 
set $S \subset H$ of size at most $n$, there exists some $f \in F_n$ for which $S \subset f A$.
Now, let $T \subset G$ of size at most $n$.  We can decompose $T = \bigsqcup_{r \in R} S_r r$
for some sets $S_r \subset H$ for each $r \in R$.  Consider $S = \bigcup_{r\in R} S_r$.
Since 
$$ |S| \leq \sum_{r \in R} |S_r| = \sum_{r \in R} |S_r r | = |T| \leq n , $$
there exists $f \in F_n$ such that $S_r \subset S \subset f A$ for all $r \in R$.
Hence $S_r r \subset f Ar \subset fB$ for all $r \in R$ implying that $T \subset fB$.

We conclude that if $[G:H]<\infty$ and there are $\CS$ sets in $H$ of arbitrarily small density, then 
the same property holds in $G$ as well.

So we are left with analyzing the case where $G$ is nilpotent.

Since $G/[G,G]$ is an infinite finitely generated abelian group, there exists a surjective homomorphism $\vphi: G \to \Z$.

Let $\eps>0$ and using Theorem \ref{thm:small CS in Z}
take a $\CS$ set $A \subset \Z$ of density less than $\eps$.
By Lemma \ref{lem:CS from quotient}, $\vphi^{-1}(A)$ is a $\CS$ set in $G$.
Since $G$ is nilpotent, it has polynomial growth (see \cite{Bass, Guivarch, Wolf}, 
and for more details \cite[Chapter 8.2]{HFbook}).
By Lemma \ref{lem:density}, $\vphi^{-1}(A)$ has density less than $C \eps$
for some constant $C>0$.
\end{proof}

\begin{remark}
In general, there are several ways to define density. Given an increasing sequence $\mathcal{F} = (F_n)_n$ of finite subsets 
with $\bigcup_{n=1}^\infty F_n = G$, the \define{upper density} and \define{lower density} relative to $\mathcal{F}$ of a set $A \subset G$ are defined by
\[
\limsup_{n \to \infty} \frac{|A \cap F_n|}{|F_n|} \quad \text{and} \quad \liminf_{n \to \infty} \frac{|A \cap F_n|}{|F_n|},
\]
respectively. In this paper, we take $\mathcal{F}$ to be the sequence of balls of increasing radii. Since our goal is to demonstrate the existence of $\CS$ sets with arbitrarily small positive density, we focus on the upper density; any such example automatically has equally small or smaller lower density.

While using balls is a natural choice, in the context of amenable groups it is common to define density relative to a F{\o}lner sequence. Note, however, that if $G$ is virtually nilpotent, then it has polynomial growth, and so the sequence of balls forms a F{\o}lner sequence. On the other hand, if $G$ is not virtually nilpotent, then the existence of infinitely many disjoint $\CS$ sets implies that any subadditive function from the power set of $G$ into the unit interval $[0,1]$ (and, in particular, an upper density relative to {\em any} sequence) attains arbitrarily small values on some $\CS$ set. In particular, when $G$ is amenable, our result can be reformulated in terms of upper (hence lower) densities relative to {\em some} F{\o}lner sequence (but if it is not virtually nilpotent, we do not know whether this applies to {\em all} F{\o}lner sequences).
\end{remark}

%

\def\polhk#1{\setbox0=\hbox{#1}{\ooalign{\hidewidth
  \lower1.5ex\hbox{`}\hidewidth\crcr\unhbox0}}}

\end{document}